\documentclass[12pt]{article}
\usepackage{amsmath,amsfonts,amsthm,mathrsfs}
\usepackage{enumerate}
\usepackage{authblk}

\numberwithin{equation}{section}
\newtheorem{thm}{Theorem}[section]
\newtheorem{lemma}[thm]{Lemma}
\newtheorem{pro}[thm]{Proposition}
\newtheorem{corollary}[thm]{Corollary}
\newtheorem{definition}{Definition}

\newtheorem{rem}[thm]{Remark}

\newtheorem{alg}{Algorithm}
\newtheorem*{keywords}{Keywords}

\newcommand{\be}{\begin{equation}}
\newcommand{\ee}{\end{equation}}
\newcommand{\bea}{\begin{eqnarray*}}
\newcommand{\eea}{\end{eqnarray*}}
\newcommand{\dom}{\mathrm{dom}}
\newcommand{\mR}{\mathbb{R}}
\newcommand{\mN}{\mathbb{N}}

\newcommand{\mcH}{\mathcal{H}}
\newcommand{\mcX}{\mathcal{X}}

\newcommand{\prox}{\operatorname{prox}}
\DeclareMathOperator*{\argmin}{argmin}
\newcommand{\la}{\langle}
\newcommand{\ra}{\rangle}

\newcommand{\eref}[1] {(\ref{#1})}

\begin{document}

\title{Modified Fej\'{e}r  Sequences and  Applications\footnote{Email address: jhlin5@hotmail.com; lrosasco@mit.edu; silvia.villa@iit.it; mazhou@cityu.edu.hk. This material is based upon work supported by the Center for Brains, Minds and Machines (CBMM), funded by NSF STC award CCF-1231216.
The work by D. X. Zhou described in this paper is supported by a grant from the NSFC/RGC Joint Research Scheme [RGC Project No. N\_CityU120/14 and NSFC Project No. 11461161006]. L. R. acknowledges the financial support of the Italian Ministry of Education, University and Research FIRB project RBFR12M3AC. S. V. is member of the Gruppo Nazionale per l'Analisi Matematica, la Probabilit\`a e le loro Applicazioni (GNAMPA) of the Istituto Nazionale di Alta Matematica (INdAM).}}
\renewcommand\Authfont{\footnotesize}
\renewcommand\Affilfont{\fontsize{9}{10.8}\itshape}

\author[$\dag$]{Junhong Lin}
\author[$\dag, \circ$]{Lorenzo Rosasco}
\author[$\dag$]{Silvia Villa}
\author[*]{Ding-Xuan Zhou}
\affil[$\dag$]{LCSL, Istituto Italiano di Tecnologia and Massachusetts Institute of Technology, Cambridge, MA 02139, USA}
\affil[$\circ$]{DIBRIS, Universit\'a degli Studi di Genova, Genova 16146, Italy}
\affil[*]{Department of Mathematics, City University of Hong Kong,
Kowloon, Hong Kong, China}
\maketitle \baselineskip 16pt
 \begin{abstract}
In this note, we propose and study  the notion of modified Fej\'{e}r sequences.
Within a  Hilbert space setting,  we show that it
provides a unifying framework to prove  convergence rates for objective function values of  several optimization algorithms.
In particular, our results apply to forward-backward splitting algorithm,  incremental subgradient proximal algorithm, and the Douglas-Rachford splitting method including and generalizing known results.
\end{abstract}

\begin{keywords}{\small \em Convex optimization, rate of convergence of objective function values, Fej\'er sequences, subgradient method, proximal splitting methods}
\end{keywords}

\section{Introduction}

The notion of Fej\'{e}r monotonicity captures essential properties of the iterates
generated by a wide range of optimization methods and provides a common framework to analyze their convergence \cite{Combettes2001}.
Quasi-Fej\'er monotonicity is a relaxation of the above notion that allows for  an additional error term \cite{Ermoliev1968,Combettes2009}.
In this paper, we propose and study a novel, related notion to analyze the convergence of the objective function values,
in addition to that of the iterates.
More precisely, we  modify the notion of quasi-Fej\'er monotonicity,  by adding a term involving the objective function and
 say that  a sequence satisfying the new requirement is  modified Fej\'er monotone (modified Fej\'er for short).
In this paper, we show the usefulness of this new notion of monotonicity by
 deriving convergence rates for several  optimization algorithms in a unified  way.
Based on this approach, we not only recover known results, such as the sublinear convergence rate for the proximal
forward-backward splitting algorithm, but also derive new results. Interestingly, our results show that for projected subgradient, incremental subgradient proximal, and Douglas-Rachford algorithm,  considering the last iteration
  leads to  essentially the same convergence rate as  considering the best
iterate selection rule \cite{Shor1979,Polyak1987},  or ergodic means \cite{Boyd2003,Singer2009},  as typically done.

\section{Modified Fej\'{e}r Sequences}
Throughout this paper, we assume that $f: \mcH \to \left]-\infty,\infty\right]$ is a proper function on $\mcH$. Assume that the set
\[
\mcX=\{z\in\mcH\,|\,f(z)=\min_{x\in \mcH} f(x) \}
\]
is nonempty. We are interested in solving the following optimization problem
\be\label{GeneralOptimization}
f_* = \min_{x \in \mcH} f(x).
\ee
Given $x \in \mcH$ and a subset $S\subset \mcH$,  $d(x,S)$ denotes  the distance between $x$
and $S,$ i.e., $d(x,S) = \inf_{x' \in S} \|x - x'\|.$ $\mR_{+}$ is the set of all non-negative real numbers
and $\mN^*=\mN\setminus\{0\}$.
For any $S\subset \mcH$, we denote by ${\mathbf 1}_{\{\cdot\}} $ the characteristic function of $S$.

\noindent The following  definition introduces the key notion we propose in  this paper.

\begin{definition}\label{Def:GFejer}
 A sequence $\{x_t\}_{t \in \mN}\subset\mcH$ is modified Fej\'{e}r monotone with respect to the
 target function $f$ and the sequence $\{(\eta_t, \xi_t)\}_{t \in \mN}$ in $\mR_{+}^2 $, if 
  \be\label{GFejerSequn}
(\forall x \in \dom f)\qquad  \|x_{t+1} - x \|^2 \leq \|x_t - x\|^2  - \eta_t ( f(x_t) - f(x)) + \xi_t . 
  \ee
\end{definition}
\begin{rem}\ 
\begin{itemize}
\item[(i)] Choosing $x \in \mcX$ in \eref{GFejerSequn}, we get $$\eta_t f(x_t) \leq \xi_t + \eta_t f_* + \|x_t - x\|^2 < \infty.$$
This implies that $\{x_t\}_{t\in\mN}\subset \dom f$.
\item[(ii)] By setting $x = x_t$ in \eref{GFejerSequn}, a direct consequence  is that,
for all $t \in \mN$,
\be\label{Difference}
\|x_{t+1} - x_t \|^2 \leq \xi_t.
\ee
\item[(iii)] All the subsequent results hold if condition \eqref{GFejerSequn} is replaced by the following weaker condition
  \be
  \label{GFejerSequn2}
(\forall x\in\mcX\cup\{x_t\}_{t\in\mN})\quad  \|x_{t+1} - x \|^2 \leq \|x_t - x\|^2  - \eta_t ( f(x_t) - f(x)) + \xi_t . 
  \ee
However, in the proposed applications, condition \eqref{GFejerSequn} is always satisfied for every $x\in\dom f$.
\end{itemize}
\end{rem}
In the following remark we discuss the relation with classical Fej\'er sequences.

\begin{rem}[Comparison with quasi-Fej\'er sequences] \ \\
Let $C$ be a nonempty subset of $\mcH$. If $\sum_{t\in\mN} \xi_t<+\infty$,  Definition~\ref{Def:GFejer} applied to the function $f+\iota_C$,  implies that the sequence
$\{x_t\}_{t\in\mN}$ is quasi-Fej\'er monotone with respect to $C$ \cite{Ermoliev1968,Combettes2009}. Indeed, \eqref{GFejerSequn} implies
$$(\forall x \in C)\qquad \|x_{t+1} - x \|^2 \leq \|x_t - x\|^2  + \xi_t .$$
Note that, in the study of convergence properties of quasi-Fej\'er sequences corresponding to a minimization problem,
the property is considered with respect to the set of solutions $\mcX$, while here we will consider modified Fej\'er monotonicity
 for a general constraint set or the entire space $\mcH$.
\end{rem}

We next present  two main results to show how modified Fej\'er sequences are useful to study the convergence of optimization  algorithms. The first result shows that if a sequence is  modified Fej\'{e}r monotone, one can bound its corresponding excess function
values in terms of $\{(\eta_t, \xi_t)\}_{t \in \mN}$ explicitly.
\begin{thm}\label{Thm:ErrorDecomp}
Let $\{ x_t \}_{t \in \mN}\subset \mcH$ be a modified Fej\'{e}r sequence with respect to $f$
and $\{(\eta_t, \xi_t) \}_{t \in \mN}$ in  $\mR_{+}^2$. Let $\{\eta_t\}_{t \in \mN}$ be a non-increasing sequence. Let $T\in\mN$, $T> 1$. Then
\begin{align}
\label{GeneralBound}
\eta_T (f(x_{T}) - f_*)   \leq & {1 \over T} d(x_1, \mcX)^2 + \sum_{t=1}^{T-1}\frac{1}{T-t}\xi_t + \xi_T.
\end{align}
\end{thm}

\begin{proof}
Let   $\{u_j\}_{j\in\mN}$ be a sequence in $\mR$ and let $k\in\{1, \cdots, T-1\}$. We have
  \[
  \begin{split}
  & {1 \over k} \sum_{j=T-k+1}^{T} u_{j} - {1 \over k+1} \sum_{j=T-k}^T u_j \\
   =& {1 \over k(k+1)} \left\{ (k+1)\sum_{j=T-k+1}^{T} u_{j} - k \sum_{j=T-k}^T u_j \right\} \\
  =& {1 \over k(k+1)} \sum_{j=T-k+1}^{T} (u_{j} -  u_{T-k}) .
  \end{split}
  \]
Summing over $k=1, \cdots, T-1$, and rearranging terms, we get
\be\label{BasicErrordecomposition}
   u_{T}  = {1 \over T} \sum_{j=1}^T u_j + \sum_{k=1}^{T-1} {1 \over k(k+1)} \sum_{j=T-k+1}^{T} (u_{j} -  u_{T-k}) .
  \ee
For any $x \in \dom f$, choosing $(\forall t\in\mN)\; u_t=\eta_t(f(x_t)-f(x))$ and rearranging terms,
we have the following error decomposition \cite{Lin2015}:
\begin{eqnarray}
&& \eta_T (f(x_{T}) - f(x)) = {1 \over T} \sum_{t=1}^T \eta_t (f(x_t) - f(x)) \nonumber \\
&& + \sum_{k=1}^{T-1} \frac{1}{k(k+1)} \sum_{t=T-k+1}^{T} \eta_t (f(x_t) - f(x_{T-k})) \nonumber \\
&& + \sum_{k=1}^{T-1} \frac{1}{k+1} \left[\frac{1}{k}\sum_{t=T-k+1}^{T} \eta_{t} - \eta_{T-k}\right] \left\{f(x_{T-k}) - f(x)\right\} \nonumber.
\end{eqnarray}
Let $x=x_* \in \mcX$. Since $\{\eta_t\}_{t \in \mN}$ is a non-increasing sequence and $f(x_{T-k}) - f_* \geq 0$,
the last term of the above inequality is non-positive. Thus, we derive that
\begin{eqnarray}
&& \eta_T (f(x_{T+1}) - f_*) \leq {1 \over T} \sum_{t=1}^T \eta_t (f(x_t) - f(x_*)) \nonumber \\
&& + \sum_{k=1}^{T-1} \frac{1}{k(k+1)} \sum_{t=T-k+1}^{T} \eta_t (f(x_t) - f(x_{T-k})). \label{ErrorDecompo}
\end{eqnarray}
For every $j\in\{1,\ldots,T\}$, and for every $x\in \dom f$,  summing up (\ref{GFejerSequn}) over $t = j, \cdots,T$,
we get
\be\label{GFejerSequnInterm}
\sum_{t=j}^{T} \eta_t ( f(x_t) - f(x)) \leq \|x_j - x\|^2 + \sum_{t=j}^{T} \xi_t.
\ee
The above inequality with $x= x_*$ and $j=1$ implies
\begin{equation}
\label{e:uno}
{1 \over T} \sum_{t=1}^T \eta_t (f(x_t) - f(x_*)) \leq {1 \over T} \|x_1 - x_*\|^2  + {1 \over T} \sum_{t=1}^{T} \xi_t.
\end{equation}
Inequality \eref{GFejerSequnInterm} with $x= x_{T-k}$ and $j=T-k$ yields
\begin{eqnarray}
 &&\sum_{k=1}^{T-1} \frac{1}{k(k+1)} \sum_{t=T-k+1}^{T} \eta_t (f(x_t) - f(x_{T-k})) \nonumber\\
 &=& \sum_{k=1}^{T-1} \frac{1}{k(k+1)} \sum_{t=T-k}^{T} \eta_t (f(x_t) - f(x_{T-k})) \nonumber\\
  &\leq&  \sum_{k=1}^{T-1} \frac{1}{k(k+1)} \sum_{t=T-k}^{T} \xi_t. \label{e:tre}
\end{eqnarray}
Exchanging the order in the sum, we obtain
\begin{align}
\label{e:due}
\nonumber\sum_{k=1}^{T-1} \frac{1}{k(k+1)} \sum_{t=T-k}^{T} \xi_t& =\sum_{t=1}^{T-1}\sum_{k=T-t}^{T-1} \frac{1}{k(k+1)} \xi_t +\sum_{k=1}^{T-1}\frac{1}{k(k+1)} \xi_T\\
\nonumber&=\sum_{t=1}^{T-1} \left(\frac{1}{T-t}-\frac{1}{T}\right) \xi_t +\left(1-\frac{1}{T}\right)\xi_T\\
&=\sum_{t=1}^{T-1} \frac{1}{T-t} \xi_t +\xi_T-\frac{1}{T}\sum_{t=1}^{T}\xi_t.
\end{align}
The result follows by plugging \eqref{e:uno},\eqref{e:tre}, and \eqref{e:due} into (\ref{ErrorDecompo}).
\end{proof}
In the special case when, for every $t\in\mN$, $\xi_t=0$, we derive the following result.
\begin{corollary}\label{Cor:xi0}
Let $\{ x_t \}_{t \in \mN} \subset \mcH$ be a modified Fej\'{e}r sequence with respect to  $f$
and a sequence $\{(\eta_t, \xi_t)\}_{t \in \mN}$ in $\mR^2_+$. Assume that $\xi_t=0$  for every $t\in \mN$,
and $\{\eta_t\}_{t\in\mN}$ is non-increasing. Let $T\in\mN$, $T>1$. Then
$$
  f(x_{T}) - f_*   \leq  {1 \over \eta_T T } d(x_1, \mcX)^2  .
$$
\end{corollary}

The second main  result shows how to derive explicit rates for the objective function values corresponding to
a modified Fej\'{e}r sequence with respect to polynomially decaying sequences $\{(\eta_t, \xi_t) \}_{t \in \mN}$ in $\mR_{+}^2 $.
Interestingly,  the following result (as well as the previous ones) does not require convexity of $f$.

\begin{thm}\label{Thm:GeneralConvergence}
Let $\{ x_t \}_{t \in \mN} \subset C$ be a  modified Fej\'{e}r sequence with respect to a target function $f$
  and $\{(\eta_t, \xi_t)  \}_{t \in \mN}\subset\mR_{+}^2$. Let $\eta\in\left]0,+\infty\right[$, let $\theta_1 \in \left[0,1\right[$,  and set $\eta_t = \eta t^{-\theta_1}$.
 Let $(\theta_2,\xi)\in\mR^2_+$ and suppose that $\xi_t \leq \xi t^{- \theta_2}$ for all $t \in \mN$.
Let $T\in\mN$, $T\geq 3$. Then
\be
f(x_{T}) - f_* \leq {d(x_1, \mcX)^2 \over \eta}  T^{\theta_1-1}  + {\xi c_{\theta_2} \over \eta} (\log T)^{\mathbf{1}_{\{\theta_2 \leq 1\}}} T^{\theta_1 - \min\{\theta_2,1\}}.
\ee
Here,
\begin{equation}
\label{e:defc}
c\colon\mR_+\to \mR_+, \quad \theta_2\mapsto c_{\theta_2}=
\begin{cases}
5+ \dfrac{2 }{ 1 - \theta_2} & \text{if } \theta_2 < 1,\\[1.1ex]
   9                                                            & \text {if } \theta_2 =1, \\[1.1ex]
\dfrac{2^{\theta_2}+3\theta_2-1 }{ \theta_2 - 1} & \text{if } \ \theta_2 > 1.
\end{cases}
\end{equation}
\end{thm}
To prove this result, we will use Theorem \ref{Thm:ErrorDecomp} as well as the following lemma.
\begin{lemma}\label{Lem:EstimatingTerm2}
Let  $q \in \mR_+$ and $T\in\mN$, $T\geq 3$. Then
  \bea  \sum_{t=1}^{T-1} {1 \over T-t} t^{-q}
 \leq
 \left\{ \begin{array}
   {ll}
   \left(4+ {2 /( 1-q)}\right) T^{- q} \log T, & \hbox{when} \ q < 1,\\
   {8 T^{-1} \log T},                          & \hbox{when} \ q = 1,\\
   (2^q+2q)/(q-1)T^{-1},       & \hbox{when} \ q >1, \\
 \end{array}\right.
  \eea
\end{lemma}

\begin{proof}
We  split the sum into two parts
 \bea
 \sum_{t=1}^{T-1} {1 \over T-t} t^{-q} &=& \sum_{T/2 \leq t \leq T -1 }  {1 \over T-t} t^{-q} + \sum_{1\leq t < T/2}  {1 \over T-t} t^{-q} \\
 & \leq & 2^q T^{-q} \sum_{T/2 \leq t \leq T-1}  {1 \over T-t} + 2 T^{-1} \sum_{1\leq t < T/2}   t^{-q} \\
 & = & 2^q T^{-q} \sum_{1 \leq t \leq T/2}   t^{-1} + 2 T^{-1} \sum_{1\leq t < T/2}   t^{-q}.
\eea
Applying, for $T\geq 3$,
 $$ \sum_{t=1}^{T} t^{-\theta_2} \leq 1 + \int_1^T u^{-\theta_2} d u \leq
 \left\{ \begin{array}
   {ll}
   {T^{1 - \theta_2}/(1-\theta_2)}, & \hbox{when} \ \theta_2 < 1,\\
   {2 \log T},                          & \hbox{when} \ \theta_2 = 1,\\
   {\theta_2/(\theta_2 - 1)},       & \hbox{when} \ \theta_2 >1, \\
 \end{array}\right.
 $$
 we get
 $$ \sum_{t=1}^{T-1} {1 \over T-t} t^{-q}  \leq 2^{q+1} T^{-q} \log T +
 \left\{ \begin{array}
   {ll}
   (2/(1-q))T^{ - q}, & \hbox{when} \ q < 1,\\
   {4 T^{-1} \log T},                          & \hbox{when} \ q = 1,\\
   {2q} T^{-1}/(q-1),       & \hbox{when} \ q >1, \\
 \end{array}\right.
 $$
which leads to the desired result by using  $T^{-q+1} \log T  \leq 1/(2(q-1))$ when $q> 1$.
\end{proof}

Now, we are ready to prove Theorem \ref{Thm:GeneralConvergence}.
\begin{proof}[Proof of Theorem \ref{Thm:GeneralConvergence}]
It follows from Theorem \ref{Thm:ErrorDecomp} that \eref{GeneralBound} holds.
Substituting  $\eta_t = \eta t^{- \theta_1} $, $\xi_t \leq \xi t^{-\theta_2}$,
 \bea
  \eta T^{-\theta_1} (f(x_{T}) - f_*) &\leq& {1 \over T} d(x_1, \mcX)^2 + \xi \sum_{t=1}^{T-1}\frac{1}{T-t} t^{-\theta_2}  + \xi T^{-\theta_2}.
 \eea
Lemma \ref{Lem:EstimatingTerm2} yields
 \bea
  \eta T^{-\theta_1} (f(x_{T}) - f_*) \leq {1 \over T}  d(x_1, \mcX)^2 + \xi {c}_{\theta_2} (\log T)^{\mathbf{1}_{\{\theta_2 \leq 1\}}} T^{- \min\{\theta_2,1\}}.
 \eea
  The results follows dividing both sides by $\eta T^{-\theta_1}$.
\end{proof}

\section{Applications in Convex Optimization}
In this section, we apply previous results to some convex optimization algorithms, including forward-backward splitting, projected subgradient, incremental proximal subgradient,  and Douglas-Rachford splitting methods.
Convergence rates for the objective function values are obtained by using Theorem \ref{Thm:GeneralConvergence}.
The key observation is that the sequences generated by these algorithms are modified Fej\'{e}r monotone.

Throughout this section, we assume that $\mcH$ is a Hilbert space, and $f: \mcH \to ]-\infty,\infty]$ is a proper, lower semicontinuous convex function.
Recall that the subdifferential of $f$ at $x\in\mcH$ is
\be\label{Subdifferential} \partial f(x) = \{u\in \mcH:\,(\forall y\in\mcH)\,\, \  f(x) + \la u, y-x\ra \leq f(y) \}. \ee
The elements of the subdifferential of $f$ at $x$ are called subgradients of $f$ at $x$.
More generally, for $\epsilon\in\left]0,+\infty\right[$, the $\epsilon$-subdifferential of $f$ at $x$ is the set $\partial_{\epsilon} f(x)$ defined by
\be
\label{e:esubd}
\partial_{\epsilon} f(x) = \{u\in \mcH: \,(\forall y\in\mcH)\,\, \ f(x) + \la u, y-x\ra - \epsilon \leq f(y) \}.
\ee
The proximity operator of $f$ \cite{Mor62} is
\begin{equation}
\label{e:prox}
\prox_f(x)=\argmin_{y\in\mcH}\left\{f(y)+\frac 1 2 \|y-x\|^2 \right\}.
\end{equation}

\subsection{Forward-Backward Splitting}\label{Subsec:FoBos}
In this subsection, we consider a forward-backward splitting algorithm for solving Problem~\eref{GeneralOptimization}, with objective 
function
\begin{equation}
\label{e:sum}
f = l + r
\end{equation}
where $r\colon\mR\to\left]-\infty,\infty\right]$ and  $l\colon\mcH\to\mR$ are proper, lower semicontinuous, and convex.
Since $l$ is real-valued, we have  $\dom\, \partial l=\mcH$ \cite[Proposition 16.14]{Bauschke2011}. 
\begin{alg}
  \label{AlgFoBos}
  Given $x_{1} \in \mcH$, a sequence of stepsizes $\{\alpha_t \}_{t \in \mN}\subset \left]0,+\infty\right[$,
  and a sequence $\{\epsilon_t\}_{t \in \mN}\subset \left[0,+\infty\right[$
  set, for every $t\in\mN$,
  \be
  \label{FoBos}
     x_{t+1} = \prox_{\alpha_tr}(x_t-\alpha_tg_t)
 \ee
with $g_t\in \partial_{\epsilon_t} l(x_t)$.
\end{alg}

The forward-backward splitting algorithm has been well studied \cite{Tseng1991,Chen1997,Combettes2005,Bredies2008}
and a review of this algorithm can be found in \cite{Combettes2011} under the assumption that $l$ is differentiable
with a Lipschitz continuous gradient.
Convergence is proved using arguments based on Fej\'{e}r monotonicity of the generated sequences \cite{Combettes2009}.
Under the assumption that $l$ is a differentiable function with Lipschitz continuous gradient, the algorithm exhibits a
sublinear convergence rate $O(T^{-1})$ on the objective $f$ \cite{Beck2009}.
If  $l$ is not smooth, the algorithm has been studied first in \cite{Passty1979},
and has a convergence rate $O(T^{-1/2})$, considering the best point selection rule \cite{Singer2009}.
Our objective here is to provide a convergence rate for the algorithm considering the last iteration,  which shares the same rate
(up-to logarithmic factors)  and to allow the use of $\epsilon$-subgradients, instead of subgradients.

\begin{thm}
\label{thm:ratefb}
Let $\alpha\in\left]0,+\infty\right[$,   let $\theta \in \left]0,1\right[$, and let, for every $t\in\mN^*$,  $\alpha_t = \alpha t^{-\theta}$.
Let $\epsilon\in\left]0,+\infty\right[$, $\{\epsilon_t\}_{t\in\mN^*}\subset \left[0,+\infty\right]$,
  and assume that $\epsilon_t\leq \epsilon\alpha_t$.
Let $\{x_t\}_{t\in\mN^*}$ be the sequence generated by Algorithm~\ref{AlgFoBos}.
Assume that there exists $B\in\left]0,+\infty\right[$ such that
\begin{equation}
\label{e:bdsgd} (\forall g \in \partial l_{\epsilon_t}(x_t) \cup \partial r(x_{t})) \quad \|g\| \leq B,
\end{equation}
 and let $c$ be defined as in \eqref{e:defc}.  Let $T\in\mN$, $T>3$. Then
 $$ f(x_{T}) - f_* \leq  {d(x_1, \mcX)^2 \over 2\alpha} T^{\theta-1}  + \alpha(5B^2+\epsilon)  c_{2\theta} (\log T)^{\mathbf{1}_{\{2\theta \leq 1\}}} T^{- \min\{\theta,1 -\theta\}} .$$
\end{thm}

\begin{proof}
Let $t\in\mN^*$. By Fermat's rule (see e.g. \cite[Theorem 16.2]{Bauschke2011}), $$0 \in x_{t+1} - x_{t} + \alpha_t g_t + \alpha_t \partial r(x_{t+1}).$$
Thus, there exists $q_{t+1} \in \partial r(x_{t+1})$, such that  $x_{t+1}$ in \eref{FoBos} can be written as
\be\label{FoBosEquivalent}
x_{t+1} = x_t - \alpha_t g_t-\alpha_t q_{t+1}.\qquad
\ee
Note that $\{x_t\}_{t\in\mN^*} \subset \dom f$ and let $x\in\dom f$.
Using \eref{FoBosEquivalent} and expanding $\|x_{t+1} - x\|^2,$ we get
\be\label{e4}
\|x_{t+1} - x\|^2 = \|x_t - x\|^2 + \alpha_t^2 \|g_t+ q_{t+1}\|^2 - 2\alpha_t \la x_t - x, g_t \ra - 2\alpha_t \la x_t - x, q_{t+1}\ra.
\ee
By \eqref{e:bdsgd},
\be\label{e2}
 \alpha_t^2 \|g_t+ q_{t+1}\|^2 \leq 4 \alpha_t^2 B^2.
  \ee
By \eref{e:esubd},
\be\label{e3}
 \la x_t - x, g_t \ra  \geq l(x_t) - l(x)-\epsilon_t,
 \ee
and convexity of $r$ implies
\bea
\la x_{t} - x,q_{t+1} \ra &=& \la x_{t} - x_{t+1}, q_{t+1} \ra +\la x_{t+1} - x, q_{t+1} \ra \\
&\geq& \la x_{t} - x_{t+1}, q_{t+1} \ra + r(x_{t+1}) - r(x).
\eea
Using \eref{FoBosEquivalent} and then applying Cauchy inequality,
\bea
\la x_{t} - x, q_{t+1} \ra  &\geq& \la x_{t} - x_{t+1}, q_{t+1} \ra + r(x_{t+1}) - r(x)\\
&=& \alpha_t \la g_t, q_{t+1}\ra + \alpha_t \| q_{t+1}\|^2 + r(x_{t+1}) - r(x) \\
&\geq& - \alpha_t \| g_t\| \| q_{t+1}\| + r(x_{t+1}) - r(x)\\
&\geq& -\alpha_t B^2 + r(x_{t+1}) - r(x) \\
&=& -\alpha_t B^2 + [r(x_{t}) - r(x)] + [r(x_{t+1}) - r(x_t)].
\eea
Let $q_t\in\partial r(x_t)$. By convexity,  $ r(x_{t+1}) - r(x_t) \geq \la x_{t+1} - x_t, q_t \ra$.
Moreover, recalling the expression in \eref{FoBosEquivalent}, we get
\begin{align}
\label{es}
\nonumber\la x_{t} - x, q_{t+1} \ra
\nonumber&\geq -\alpha_t B^2 + [r(x_{t}) - r(x)] + \la x_{t+1} - x_t , q_{t} \ra \\
\nonumber&=  -\alpha_t B^2 + [r(x_{t}) - r(x)] - \alpha_t \la g_t + q_{t+1} , q_t\ra \\
\nonumber&\geq  -\alpha_t B^2 + [r(x_{t}) - r(x)] -  \alpha_t(\|g_t\| + \|q_{t+1}\|) \|q_t\| \\
&\geq -\alpha_t B^2 + [r(x_{t}) - r(x)] - 2 \alpha_t B^2.
\end{align}
It follows from \eref{e4}, \eref{e2}, \eref{e3}, and \eqref{es}  that
\bea
\|x_{t+1} - x\|^2 &\leq& \|x_t - x\|^2 - 2\alpha_t [l(x_t) - l(x)] - 2\alpha_t [r(x_{t}) - r(x)]  + 10\alpha_t^2 B^2+2\alpha_t\epsilon_t \\
&=& \|x_t - x\|^2 - 2\alpha_t [f(x_t) - f(x)] +\alpha_t^2 (10B^2+2\epsilon)
\eea
Thus, $\{ x_t \}_{t\in\mN^*}$ is a modified Fej\'{e}r sequence with respect to the target function $f$ and $\{(2\alpha_t, (10  B^2+2\epsilon) \alpha_t^2) \}_{t\in\mN^*}$.
The statement follows from Theorem \ref{Thm:GeneralConvergence}, applied with $\theta_1 = \theta$, $\theta_2 = 2\theta,$ $\eta= 2 \alpha$ and $\xi
= (10 B^2 +2\epsilon)\alpha$.
\end{proof}
The following remark collects some comments on the previous result.
\begin{rem}\
\begin{enumerate}
\item Setting $\theta=1/2$, we get a convergence rate $O(T^{-1/2} \log T)$ for forward-backward algorithm with nonsummable diminishing stepsizes, considering the last iteration.
\item In Theorem \ref{thm:ratefb}, the assumption on bounded approximate subgradients, which is equivalent to Lipschitz continuity of $l$ and $r$,
 is satisfied for some practical optimization problems. For example, when $r$ is the indicator function of a closed, bounded, and convex set  $D\subset \mR^N$, 
it follows that $\{x_t\}_{t\in\mN}$ is bounded, which implies $\{g_t\}_{t\in\mN}$ is bounded as well \cite{Alber1998}.
 For general cases, similar results may be obtained by imposing a growth condition on $\partial f$, using a similar approach to that in \cite{Lin2015} to
bound the sequence of subgradients.
\end{enumerate}
\end{rem}
If the function $l$ in \eqref{e:sum} is differentiable, with a Lipschitz differentiable gradient, we recover the well-known $O(1/T)$ convergence rate for
the objective function values.
\begin{pro}{\cite[Theorem 3.1]{Beck2009}}
\label{prop:fbdiff}
Let $\beta\in\left[0,+\infty\right[$ and assume that $\nabla l$ is $\beta$-Lipschitz continuous. Consider Algorithm~\ref{AlgFoBos} with $\epsilon=0$  and $\alpha_t= \beta$ for all $t\in \mN^*$.
Then, for every $T\in\mN$, $T>1$
\begin{equation}
\label{e:fbdff}
f(x_t)-f_*\leq \frac{\beta d(x_1,\mcX)^2}{2 T}
\end{equation}
\end{pro}
\begin{proof}
It follows from \cite[Equation 3.6]{Beck2009} that
\begin{equation}
\label{e:e36}
(\forall t\in\mN^*)\quad \frac{2}{\beta} (f(x_t)-f_*)\leq \|x_{t+1}-x_*\|^2-\|x_{t}-x_*\|^2.
\end{equation}
Thus, $\{x_t\}_{t\in\mN^*}$ is a modified Fej\'er sequence with respect to the target function $f$
and the sequence $\{(\eta_t,\xi_t)\}_{t\in\mN^*}$ with $(\forall t\in\mN)\,\,\eta_t=2/\beta$ and $\xi_t=0$.
The statement follows from Corollary~\ref{Cor:xi0}.
\end{proof}

\subsection{Projected approximate subgradient method}
\label{sec:projsubgr}
Let $D$ be a convex and closed subset of $\mcH$, and let $\iota_D$ be the indicator function of
$D$. In this subsection, we consider Problem~\eref{GeneralOptimization} with objective function
given by
\begin{equation}
\label{e:sump}
f = l + \iota_D
\end{equation}
where $l\colon\mcH\to\mR$ is  proper, lower semicontinuous, and convex.
It is clear that \eqref{e:sump} is a special case of \eqref{e:sum} corresponding to a
given choice of $r$. The forward-backward algorithm in this case reduces to the following
projected subgradient method (see e.g. \cite{Shor1979,Polyak1987,Boyd2003} and references therein),
which allows to use $\epsilon$-subgradients, see \cite{Alber1998,Combettes2001}.
\begin{alg}
  \label{AlgProSub}
  Given $x_{1} \in \mcH$, a sequence of stepsizes $\{\alpha_t \}_{t \in \mN}\subset \left]0,+\infty\right[$,
  and a sequence $\{\epsilon_t\}_{t \in \mN}\subset \left[0,+\infty\right[,$
  set, for every $t\in\mN$,
  \be
  \label{FoBos}
     x_{t+1} =P_D(x_t-\alpha_tg_t)
 \ee
with $g_t\in \partial_{\epsilon_t} l(x_t)$.
\end{alg}
The algorithm has been studied using different rules for choosing the stepsizes.
Here,  as a corollary of Theorem~\ref{thm:ratefb}, we derive the convergence rate for the objective function values, for a nonsummable
diminishing stepsize.

\begin{thm}\label{Thm:ProjeSubgradient}
For some $\alpha_1>0$, $\epsilon \geq 0$ and $\theta \in [0,1)$, let $\alpha_t = \eta_1 t^{-\theta}$ and $\epsilon_t \leq \epsilon \alpha_t$ for all $t \in \mN^*$.
  Let $\{x_t\}_{t\in\mN}$ be a sequence generated by Algorithm \ref{AlgProSub}. Assume that for all $t \in \mN^*,$ $\|g_t\| \leq B.$ Then, for every $T\in\mN$, $T>3$
  $$ f(x_{T}) - f^* \leq  {d(x_1, \mcX)^2 \over 2\alpha_1}  T^{\theta-1}  + {\alpha_1 ( B^2 + 2 \epsilon) \widetilde{c}_{2\theta}} (\log T)^{\mathbf{1}_{\{2\theta \leq 1\}}} T^{- \min(\theta,1 -\theta)} $$
\end{thm}

Choosing $\theta = 1/2$, we get a convergence rate of order $O(T^{-1/2} \log T)$
for projected approximate subgradient methods with nonsummable diminishing stepsizes,
which is optimal up to a $\log$ factor without any further assumption on $f$  \cite{Darzentas1984,Nesterov2004}.
Since the subgradient method is not a descent method, a common approach keeps track of the
best point found so far, i.e., the one with smallest function value:
$$ (\forall T\in\mN^*)\qquad b_T = \argmin_{1\leq t \leq T} f(x_t).$$
Projected subgradient method with diminishing stepsizes of the form $\{\alpha t^{-\theta}\}_t$, with $\theta\in\left]0,1\right]$,
satisfies $b_T-f_*=O(T^{-1/2})$.
Our result shows that considering the last iterate for projected approximate subgradient method essentially
leads to the same convergence rate, up to a logarithmic factor, as the one corresponding to the best iterate,
even in the cases that the function value may not decrease at each iteration.
To the best of our knowledge, our result is the first of this kind, without any assumption on
strong convexity of $f$, or on a conditioning number with respect to subgradients (as in
\cite{Goffin1977} using stepsizes $\{\gamma_t/ \|g_t\|\}_t$).
Note that, using nonsummable diminishing stepsizes, convergence rate $O(T^{-1/2})$ was shown,
 but only for a subsequens of $\{x_t\}_{t\in\mN^*}$ \cite{Alber1998}.
Finally, let us mention that using properties of quasi-Fej\'{e}r sequences,
convergence properties were proved in \cite{Combettes2001}.

\subsection{Incremental Subgradient Proximal Algorithm}
In this subsection, we consider an incremental  subgradient proximal algorithm \cite{Bertsekas2011a,Nedic2001} for solving \eref{GeneralOptimization}, 
with objective function $f$ given by, for some $m\in\mN^*$,  $$\sum_{i=1}^{m} (l_i+r_i), $$
where for each $i$, both $l_i: \mcH \to \mR$ and $r_i:\mcH\to\left]-\infty,+\infty\right]$ are convex, proper,
and lower semicontinuous.
The algorithm is similar to the proximal subgradient method, the main difference being that
at each iteration, $x_t$ is updated incrementally, through a sequence of $m$ steps.
\begin{alg}
  \label{Alg:ProjeIGD}
Let $t\in\mN^*$. Given $x_{t} \in \mcH$, an iteration of the incremental proximal subgradient algorithm
generates $x_{t+1}$ according to the recursion,
\be\label{IGD1}
x_{t+1} = \psi_t^m,
\ee
where $\psi_t^m$ is obtained at the end of a cycle, namely as the last step of the recursion
\be\label{IGD2}
\psi_t^0 = x_t,\qquad \psi_t^i =  \prox_{\alpha_t r_i} (\psi_t^{i-1} - {\alpha_t} g_t^i), \qquad \forall g_t^i \in \partial l_i (\psi_t^{i-1}), \qquad i=1,\cdots,m
\ee
for a suitable sequence of stepsizes $\{\alpha_t\}_{t \in \mN^*}\subset\left]0,+\infty\right[$.
\end{alg}
Several versions of incremental subgradient proximal algorithms have been studied in \cite{Bertsekas2011a},
where convergence results for various stepsizes rules and both for stochastic of cyclic selection of the components
are given. Concerning the function values,  the  results are stated in terms of the best iteration.
See also \cite{Nedic2001A} for the study of the special case  of incremental subgradient methods
under different stepsizes rules. The paper \cite{Kiwiel2004} provides convergence results using
approximate subgrdients instead of gradients.

In this section, we derive a sublinear convergence rate for the incremental subgradient proximal algorithm in a straightforward
way, relying on the properties of modified Fej\'er sequences assuming a boundedness assumption on the subdifferentials, already
used in \cite{Nedic2001A}.

\begin{thm}
  \label{Thm:ProjeIGD}
   Let $\alpha\in\left]0,+\infty\right[$, let $\theta \in \left]0,1\right[$, and let, for every $t\in\mN^*$, $\alpha_t = \alpha t^{-\theta}$.
  Let $\{x_t\}_{t\in\mN^*}$ be the sequence generated by Algorithm~\ref{Alg:ProjeIGD}.
 Let $B\in\left]0,+\infty\right[$ be such that
\[
(\forall t\in\mN^*)(\forall g\in \partial l_i(x_t) \cup \partial r_i(x_t)) \qquad \|g\|\leq B,
\]
and let $c$ be defined as in \eqref{e:defc}.
  Then,   for every $T\in\mN^*$,
 $$ f(x_{T}) - f_* \leq  {d(x_1, \mcX)^2 \over 2\alpha}  T^{\theta-1}  + \frac{\alpha(4m+{5}) mB^2}{2}  {c}_{2\theta} (\log T)^{\mathbf{1}_{\{2\theta \leq 1\}}} T^{- \min\{\theta,1 -\theta\}} .$$
\end{thm}
\begin{proof}
  It was shown in \cite[Proposition 3 (Equation 27)]{Bertsekas2011a} that,
  $$   \|x_{t+1} - x\|^2 \leq \|x_t - x\|^2 - 2\alpha_t [f(x_{t}) - f(x)] + \alpha_t^2\left(4m+{5}\right) m B^2. $$
  Thus, $\{ x_t \}_{t\in\mN^*}$ is a modified Fej\'{e}r sequence with respect to the target function $f$, and $\left\{\left(2\alpha_t, \alpha_t^2\left(4m+{5}\right) mB^2\right) \right\}_{t\in\mN^*}$.
  The proof is concluded by applying Theorem \ref{Thm:GeneralConvergence} with $\theta_1 = \theta,\theta_2 = 2\theta$, $\eta = 2\alpha$ and $\xi = \alpha^2\left(4m+{5}\right) mB^2$.
\end{proof}
An immediate consequence of Theorem~\ref{Thm:ProjeIGD}, is that the choice $\theta = 1/2$ yields a convergence rate of order $O(T^{-1/2}\log T)$.

As a corollary of Theorem \ref{Thm:ProjeIGD}, we derive convergence rates for the projected incremental subgradient method. Analogously to what
we have done for the forward-backward algorithm in Section~\ref{Subsec:FoBos}, Theorem~\ref{Thm:ProjeIGD} can be extended to analyze convergence of
the  approximate and incremental subgradient method in \cite{Kiwiel2004}.
\subsection{Douglas-Rachford splitting method}
In this subsection, we consider Douglas-Rachford splitting algorithm for solving \eref{GeneralOptimization}.
Given $l\colon\mcH\to\mR$ and $r\colon\mcH\to\mR$ proper, convex, and
lower semincontinuous functions, we assume that $f=l+r$ in \eqref{GeneralOptimization}.

\begin{alg}
  \label{AlgDoRas}
Let $\{\alpha_t\}_{t\in\mN^*}\subset\left]0,+\infty\right[$.
 Let $t\in\mN^*$. Given $x_{t} \in \mcH,$ an iteration of Douglas-Rachford algorithm generates $x_{t+1}$ according to
  \be\label{DoRas}
  \left\{ \begin{array}
    {ll}
     y_{t + 1} = \prox_{\alpha_t l} (x_t)\\
     z_{t + 1} = \prox_{\alpha_t r}(2y_{t+1} - x_t), \\
     x_{t+1} = x_t + z_{t+1} - y_{t+1}.
  \end{array}
  \right.
 \ee
\end{alg}
The algorithm has been  introduced in \cite{Douglas1956} to minimize the sum of two convex functions,
and then has been extended to monotone inclusions involving the sum of two nonlinear operators \cite{Lions1979}.
A review of this algorithm can be found in \cite{Combettes2011}.
The convergence of the iterates is established  using the theory of Fej\'{e}r sequences \cite{Combettes2009}.
Our objective here is to establish a new result, namely a convergence rate for the objective function values.

\begin{thm}
  \label{thm:FoBos}
Let $\alpha\in\left]0,+\infty\right[$, and let $\theta\in\left]0,1\right[$.  For every $t\in\mN^*$, let $\alpha_t = \alpha t^{-\theta}$.
Let $\{(y_t,x_t,z_t\}_{t\in\mN^*}$ be the sequences generated by Algorithm~\ref{AlgDoRas}. Assume that there exists $B\in\left]0,+\infty\right[$
such that
$$ (\forall t\in\mN^*)(\forall g \in \partial l(y_t) \cup \partial r(z_{t}) \cup \partial l(x_t) \cup \partial r(x_{t})) \quad \|g\| \leq B.$$
Let $c$ be the function defined in \eqref{e:defc}. Then, for every $T\in\mN$, $T>3$,
 $$ f(x_{T}) - f_* \leq  {d(x_1, \mcX)^2 \over 2\alpha} T^{\theta-1}  + {8 \alpha B^2  {c}_{2\theta}} (\log T)^{\mathbf{1}_{\{2\theta \leq 1\}}} T^{- \min\{\theta,1 -\theta\}} .$$
\end{thm}
\begin{proof}
Let $t\in\mN^*$, set $v=(x_t-y_{t+1})/\alpha_t$ and $w=(2y_{t+1}-x_t-z_{t+1})/\alpha_t$.
By Fermat's rule,
\be\label{esuf}
v\in\partial l(y_{t+1}) \ \mbox{ and } w\in\partial r(z_{t+1}).
\ee
We can rewrite \eref{DoRas} as
\be\label{DoRas1}
  \left\{ \begin{array}
    {ll}
     y_{t + 1} = x_t - \alpha_t v, \\
     z_{t + 1} = (2y_{t+1} - x_t) - \alpha_t w, \\
     x_{t+1} = x_t + z_{t+1} - y_{t+1},
  \end{array}
  \right.
\ee
Thus
\be\label{DoRasEquivalent}
x_{t+1} = x_t - \alpha_t(v+w).
\ee
Using \eref{DoRasEquivalent} and expanding $\|x_{t+1} - x\|^2,$ we get
\be\label{e1}
\|x_{t+1} - x\|^2 =  \|x_{t} - x\|^2 + \alpha_t^2 \|v+w\|^2+ 2\alpha_t \la x-x_t, v \ra + 2\alpha_t \la x-x_t, w \ra.
\ee
Let $u\in\partial l(x_t)$. It follows from \eref{esuf} \eref{Subdifferential} and \eref{DoRas1} that
\bea
\la x-x_t,v \ra
&&= \la x-y_{t+1}, v \ra + \la y_{t+1} - x_t, v \ra\\
&&\leq  l(x) - l(y_{t+1}) - \alpha_t \|v\|^2 \\
&& \leq l(x) - l(x_t) + l(x_t) - l(y_{t+1}) \\
&&\leq l(x) - l(x_t) + \la x_t - y_{t+1}, u \ra \\
&&= l(x) - l(x_t) + \alpha_t \la v, u \ra \\
&& \leq l(x) - l(x_t) + \alpha_t B^2.
\eea
Similarly, Let $s\in\partial r(x_t)$. We  bound $\la x-x_t, w\ra$ as follows
\bea
\la x-x_t, w \ra
&&= \la x-z_{t+1}, w \ra + \la z_{t+1} - x_t, w\ra \\
&&\leq r(x) - r(z_{t+1}) - \alpha_t  \la 2v+ w, w\ra \\
&& \leq r(x) - r(x_t) + r(x_t) - r(z_{t+1}) + 2 \alpha_t B^2  \\
&& \leq r(x) - r(x_t) + \la x_t - z_{t+1}, s \ra + 2 \alpha_t B^2 \\
&& = r(x) - r(x_t) +  \alpha_t  \la 2v +w,s \ra +  2 \alpha_t B^2\\
&&\leq r(x) - r(x_t) + 5 \alpha_t B^2.
\eea
Introducing the above two estimates into (\ref{e1}), we get
\bea
\begin{split}
\|x_{t+1} - x\|^2 \leq & \|x_{t} - x\|^2 + 16 \alpha_t^2 B^2
 &+ 2\alpha_t (f(x) - f(x_t)) .
 \end{split}
\eea
Thus, $\{ x_t \}_{t\in\mN^*}$ is a Super Quasi-Fej\'{e}r sequence with respect to  the target function $f$ and $\{(2\alpha_t, 16 \alpha_t^2 B^2 ) \}_{t\in\mN^*}$.
The statement follows from Theorem \ref{Thm:GeneralConvergence} with $\theta_1 = \theta$ and $\theta_2 = 2\theta.$
\end{proof}
Again, choosing $\theta=1/2$, we get a convergence rate $O(T^{-1/2} \log T)$ for the algorithm with nonsummable diminishing stepsizes.
Nonergodic convergence rates for the objective function values corresponding to the Douglas-Rachford iteration can be derived by \cite[Corollary 3.5]{Davis2014},
under the additional assumption that $l$ is the indicator function of a linear subspace of $\mcH$.

\bibliographystyle{plain}

\end{document}